\documentclass[a4paper,11pt]{elsarticle}
\usepackage{amsmath,amsthm}
\usepackage{amscd}
\usepackage{amsfonts}
\usepackage{amssymb}
\usepackage{bm}                                                %
\usepackage{float}
\usepackage{graphicx,color}
\usepackage{latexsym}
\usepackage{lscape}
\usepackage{mathrsfs}
\usepackage{verbatim}
\usepackage{subcaption}
\usepackage[width=7in, height=9in]{geometry}

\newtheorem{theorem}{Theorem}[section]
\newtheorem{lemma}{Lemma}[section]
\newtheorem{remark}{Remark}[section]

\numberwithin{equation}{section}
\numberwithin{theorem}{section}
\numberwithin{remark}{section}

\begin{document}

\title{A novel Chebyshev collocation method for elliptic -type differential equations with degenerate coefficient\tnoteref{thank1}} \tnotetext[thank1]{The authors were partly supported by National Natural Science Foundation of China (12271233), Improving innovation ability of enterprises in Shandong province (2023TSGC0466).}

\author{Enze Yuan}\ead{1649772431@qq.com}

\author{Huantian Xie}\ead{xiehuantian@lyu.edu.cn}

\author{Ziwu Jiang}\ead{zwjiang@lyu.edu.cn}

\author{Jianwei Zhou\corref{cor}}\ead{jwzhou@yahoo.com}

\cortext[cor]{Corresponding author}

\address{School of Mathematics and Statistics, Linyi University, Linyi 276005, P.R. China}


\begin{abstract}
	A novel collocation scheme is presented for elliptic-type differential equations with degenerate coefficients and homogeneous Dirichlet boundary conditions. The use of weighted orthogonal Chebyshev polynomials for the basis functions leads to stiffness matrices with sparse structure, enabling efficient direct calculations. By an orthogonal projection, rigorous analyses are devoted to deriving a-priori error estimates of spectral accuracy in two norms. Furthermore, ample numerical experiments are conducted and compared with error data, convergence rates, condition numbers and $N$-$\log$ curves to confirm the theoretical analyses results. Our proposed method achieves spectral accuracy and handles boundary singularities efficiently, as demonstrated by theoretical analyses and numerical experiments.
\end{abstract}

\begin{keyword}
Chebyshev polynomial \sep degenerate coefficient \sep elliptic-type differential equation \sep  collocation method \sep a-priori error estimate.
\end{keyword}

\date{}
\maketitle
\vspace{1em}
\section{Introduction}

Degenerate differential equations, characterized by coefficients that vanish or become singular at the boundaries of the domain, constitute a class of problems that arise frequently in applied mathematics and mathematical physics. Unlike standard elliptic equations with strictly positive coefficients, these problems pose significant analytical and numerical challenges due to the loss of ellipticity and the presence of boundary singularities.
The specific model considered in this work, involving the degradation coefficient on an interval, is not merely a mathematical construct but a canonical form arising from diverse engineering and physical applications.

There are lots of representative engineering models that can be described by typical partial differential equations with degradation coefficients.
A classic example from solid mechanics involves the axial deformation of a nonuniform elastic rod. For tapered rods, idealized models are often formulated with a leading coefficient vanishes at the boundaries. This configuration leads to a second-order degenerate problem, commonly referred to as an endpoint-degenerate system. Further details can be found in \cite{Timoshenko1955}.
Meanwhile, in thermal engineering, efficient heat dissipation often relies on fins with variable profiles, such as triangular or parabolic spines  \cite{Bergman2011}. For these geometries, the cross-sectional area vanishes at the tip, causing the leading coefficient to degenerate. This creates an endpoint-degenerate boundary value problem, where standard numerical methods may lose accuracy near the singularity.

A further application emerges in fluid mechanics and tribology, particularly in modeling thin lubricant films using the Reynolds equation. In the context of elastohydrodynamic lubrication or gas-lubricated bearings in Micro-Electro-Mechanical systems, the film thickness frequently tends to zero near contact lines or sealing boundaries \cite{Hamrock2004}. Because the effective diffusion coefficient in the Reynolds equation scales with the cube of the film thickness, the governing elliptic equation becomes degenerate in these vanishing gap regions \cite{Bayada1986}. This mathematical degeneracy captures the physical phenomenon of flow stagnation and necessitates specialized numerical methods to handle the associated pressure singularities.
Complex transport phenomena in porous media represent other important source of degenerate models, notably in multiphase flow. The movement of immiscible fluids (e.g., oil and water) is described by generalized Darcy's laws, where phase mobility depends on the relative permeability. This permeability is a nonlinear function of saturation and goes to zero as the phase saturation approaches its residual limit \cite{Chen2006,Girault2021}. Consequently, the resulting governing equations for the saturation field form a degenerate parabolic system.
The authors considered a Dirichlet problem for a class of elliptic and parabolic equations, which are equipped by degenerate coefficients near the boundary. Meanwhile, they designed proper weights for the existence and uniqueness in \cite{DP2021}.
A linear degenerate elliptic model consisting of two first-order equations arising in partially melted materials, such as those in the earth's mantle, polar ice sheets, and glaciers, was presented in \cite{AT2016}. Additionally, they derived a mixed system for which the existence and uniqueness of a solution over the entire domain can be established, irrespective of degeneracy.

There is a substantial literature on numerical approximations for variable-coefficient and degenerate elliptic equations, and their analysis is often naturally posed in weighted Sobolev spaces. More details please refer to \cite{ABH2007,FKS1982} and the references cited therein.
From a computational mathematics perspective, degeneracy often arises from the choice of coordinate systems. When solving partial differential equations in $d$-dimensional polar, cylindrical, or spherical coordinates, radial operators (e.g., Laplace-Beltrami operator $\Delta = \partial_r^2 + \frac{d-1}{r}\partial_r$) exhibit apparent singularities at the origin $r=0$. These coordinate-induced singularities can be treated rigorously by reformulating the problem in weighted Sobolev spaces, as shown in \cite{STW2011}. Using generalized Jacobi polynomials or parity-preserving Fourier-Chebyshev expansions, the apparent singularity is absorbed into a weight function (e.g., $\omega(r)=r^{d-1}$), thereby transforming the originally singular problem into a well-posed degenerate variational formulation. 
Maintaining stability and accuracy near degenerate regions necessitates special treatments in classical finite difference and finite volume schemes; these include careful construction of numerical fluxes (especially at boundaries), rescaling of unknowns, and discretizations that preserve the degenerate character of the problem \cite{AT2017,EGH2000}.

Finite element methods can be formulated in weighted Sobolev spaces and provide flexible discretizations, but reduced regularity induced by degeneracy or singularity may lead to loss of optimal rates on quasi-uniform meshes unless graded or anisotropic meshes are employed \cite{ABH2007,Apel1999,NOS2015}.
Spectral and pseudospectral methods are renowned for their high-order (often exponential) accuracy when applied to generally partial differential equations with smooth solutions. These methods have been extensively developed with orthogonal polynomials such as Chebyshev, Legendre, and Jacobi bases, which provide a natural framework for achieving optimal convergence rates 
\cite{Boyd2001,CHQZ2006,STW2011}.
A combined numerical method for degenerate boundary problems are studied in \cite{HLFC2023}. And the rigorous analyses shown that the proposed approximation achieves satisfactory accuracy. 
The authors employed finite difference method with the Crank-Nicolson scheme to solve the first time inverse problems for a weakly degenerate heat equation, which vanishes at the initial moment of time in \cite{HL2020}.

Particularly, in view of the current literature, the a priori error estimates are investigated for different typical numerical approximations with degenerate cases. 
The authors in \cite{YZ2002} presented a priori error estimates for the one-dimensional compressible Navier-Stokes equations governing isentropic flow in Eulerian coordinates, specifically accounting for the degeneracy of the viscosity coefficient.
Global weighted $L^{p}$-estimates were established for the gradient of solutions to a class of linear singular, degenerate elliptic Dirichlet boundary value problems over a bounded non-smooth domain in \cite{CMP2018}. Particularly, they characterized corresponding smallness conditions for the degenerate coefficients.
Rigorous a priori error analysis for degenerate differential equations has been established within the framework of weighted Sobolev spaces \cite{Bernardi1997,Guo1998}. These works demonstrate that Chebyshev collocation methods retain their exponential convergence, even in the presence of boundary singularities, provided the solution belongs to the appropriate weighted classes. Complementary results \cite{CHQZ2006,STW2011} further confirm these estimates for models with endpoint-degenerate coefficients similar to those studied here.

For degenerate coefficients, however, it is nontrivial to design a basis that simultaneously (i) enforces boundary conditions, (ii) matches the intrinsic weighted structure, and (iii) leads to linear systems amenable to rigorous stability and approximation analysis as well as efficient solution. For example, \cite{STW2011,Shen1994} for the design of efficient Galerkin bases and solvers.

Motivated by these considerations, we develop a Chebyshev-based collocation scheme tailored to elliptic equations with endpoint-degenerate coefficients on $I:=(-1,1)$ (with endpoints $\partial I=\{-1,1\}$). 
The method is built on boundary-adapted trial functions derived from first-kind Chebyshev polynomials, so that the homogeneous Dirichlet boundary conditions are imposed exactly at the approximation level.
A key analytical ingredient is the weighted Sturm--Liouville identity, which shows that the first-kind Chebyshev polynomials satisfy an eigen-relation for the endpoint-degenerate operator. 
Combined with the orthogonality of Chebyshev polynomials in the weighted space, 
this identity allows the discrete bilinear form to be assembled in a structured manner, thereby yielding an explicitly sparse banded stiffness matrix and enabling an efficient implementation.
Unlike finite element methods that require graded meshes near singularities, our Chebyshev collocation scheme achieves high-order accuracy without mesh adaptation.

The rest of this paper is organized as follows. We begin in Section 2 by introducing the model problem with Dirichlet boundary conditions, establishing solution uniqueness via a rigorous analysis of coercivity and continuity, and developing the collocation framework with novel basis functions. In Section 3, we perform a rigorous error analysis, employing an orthogonal projection to obtain a-priori error estimates in two norms that prove spectral convergence. In the last Section, comprehensive numerical validations of the theoretical results are performed through detailed error and convergence rate studies.

\section{Model problem and its approximations}

In this section, we consider an elliptic-type problem with degenerate coefficients to demonstrate the advantages of collocation schemes based on orthogonal polynomials. The corresponding discretization system, formulated using Chebyshev polynomials, is designed to approximate the model problem. The well-posedness of the weak solution to the equivalent system is analyzed in detail.

\subsection{Model equations with variable coefficient}\label{sub-2.1} 

We focus on an elliptic-type differential equation with degenerate coefficients as follows:
\begin{equation}\label{eq:1d}
\left\{ \begin{aligned}
		-\bigl(\omega(x) \, u'(x)\bigr)' & = f(x), \quad x \in I, \\[4pt]
		u(x)|_{\partial I} & = 0,
  \end{aligned}
\right.
\end{equation}
where $\omega(x)=\sqrt{1-x^2}$ degenerates (vanishes) at the boundary of $I$ and $\int_I\frac{1}{\omega(x)}{\mathrm d}x <+\infty$. Generally, the homogeneous Dirichlet boundary conditions represent fixed endpoints in mechanical systems.

To investigate the weak solution of \eqref{eq:1d}, we must analyze this problem in a suitable Sobolev space. Obviously, the natural function space is a subspace of the weighted Sobolev space \( H_w^1(\Omega) \).
However, to guarantee the homogeneous Dirichlet boundary conditions, we define a weighted Sobolev space as the closure of $C_0^\infty(\Omega)$:
 \[
  H_{0,\omega}^1(I)= \left\{ u \in L^2(I) \;\middle|\; \int_I \omega(x) |u'(x)|^2 \, {\mathrm d}x < \infty, \ u|_{\partial I}=0 \right\}.
  \]
 
\begin{remark}
  Since the weight function \( w(x) \) is bounded and nonnegative on $I$, the \(L^2(I)\) norm is essentially unaffected by weight, or is equivalent to its weighted counterpart. For convenience, we therefore consider the space of functions \( u \in L^2(I) \) satisfying \( \sqrt{w} u' \in L^2(I) \).
\end{remark}

The function spaces are shortly denoted by
$U=L^2(I)$ and $V = H_{0,\omega}^1(I)$. And $V$ is equipped with the following full norm:
\[
 \|u\|_{H_{0,\omega}^1(I)} = \left( \int_I \left( |u|^2 + w(x)|u'|^2 \right) \, {\mathrm d}x \right)^{1/2}
\]

Indeed, the norm of $V$ is taken to be the weighted semi-norm, which is equivalent to the full norm in the sense that:
\begin{equation*}
	\|u\|_V = \left( \int_{I} \omega(x) |u'(x)|^2 \, {\mathrm d}x \right)^{1/2}.
\end{equation*}

We denote an $L^2$ inner product and a weighted $H_1$ semi-norm in $U$ and $V$, respectively, by
\[
(v, w)=\int_I v w, \quad \forall v,w\in U.
\]
and
\[
a(u, v) = \int_I \omega(x) u'(x) v'(x) {\mathrm d}x, \quad \forall v,w\in V.
\]

To investigate the equivalent weak formula of \eqref{eq:1d}, multiplying the equation \eqref{eq:1d} by any test function $v \in V$ and integrating over $I$:
\[
\int_I (\omega(x) u')' v \, {\mathrm d}x = \int_I f v \, {\mathrm d}x.
\]
Integrating by parts and using the boundary condition $v(\pm 1)=0$:
\begin{equation}\label{eq:1d-int}
\left( \omega(x) u' v \right)|_{\partial I} - \int_I \omega(x) u' v' \, {\mathrm d}x = \int_I f v \, {\mathrm d}x.
\end{equation}
The boundary term, which is listed in the first item, vanishes naturally. 

Then by \eqref{eq:1d-int}, we get the equivalent weak formulation of \eqref{eq:1d} reads: Finding $u \in V$ such that
\begin{equation}\label{eq:pde-weak}
	a(u, v) = (f, v), \quad \forall v \in V.
\end{equation}

\subsection{Coerciveness and uniqueness}

We proceed to establish the existence and uniqueness of solutions to the weak problem \eqref{eq:pde-weak}. According to the Lax-Milgram theorem, it suffices to show that the bilinear form $a(u, v)$ is bounded and coercive, and that the linear functional $(f, v)$ is bounded. Throughout, the symbols $c$ and $C$ denote various positive constants independent of the discretization parameters $N$ and $I$. A key step in analyzing the norm on $V$, which involves the $L^2(I)$ norm of $u$, is to establish a suitable weighted Poincar\'{e} inequality.

\begin{lemma}
	For any $u \in V$, there exists a constant $C > 0$ such that 
	\begin{equation}\label{poin}
		\|u\|_{U} \le C \|u\|_V .
	\end{equation}
\end{lemma}

\begin{proof}
	Since $u(-1)=0$, we can write $u(x) = \int_{-1}^x u'(t) \, {\mathrm d}t$. By the Cauchy-Schwarz inequality:
	\begin{eqnarray}
		\begin{aligned}
			|u(x)|^2 
			& \leq \left| \int_{-1}^x u'(t) \, {\mathrm d}t \right|^2
			= \left| \int_{-1}^x \omega(t)^{-1/2} \omega(t)^{1/2} u'(t) \, {\mathrm d}t \right|^2 \\
			& \leq  \left( \int_I \frac{{\mathrm d}t}{\omega(t)} \right) \left( \int_I \omega(t) |u'(t)|^2 \, {\mathrm d}t \right).
		\end{aligned}
	\end{eqnarray}
The first integral is finite. Thus, $|u(x)|^2 \leq c \|u\|_V^2$, which implies $\|u\|_U$ is bounded by $\|u\|_V$.

\end{proof}

In view of \eqref{poin}, $a(u, u) = \|u\|_V^2 \geq C \|u\|_{H^1_{0,w}}^2$ for some $C$, proving coercivity. To study boundedness of $a(u, v)$, we recall the Cauchy-Schwarz inequality:
\begin{eqnarray*}
	\begin{aligned}
	|a(u, v)| \leq \int_I |\omega(x) u'(x) v'(x)| {\mathrm d}x \le \left(\int_I \omega(x) |u'|^2 {\mathrm d}x \right)^{1/2} \left(\int_I \omega(x) |v'|^2 {\mathrm d}x \right)^{1/2} = \|u\|_V \|v\|_V.
	\end{aligned}
\end{eqnarray*}
Thus, the bilinear form is bounded (continuous). Similarly, for $f \in U$ and any $v\in V$, $(f, v)$ is bounded on $V$, obviously.

Since $V$ is a Hilbert space, the bilinear form $a(\cdot, \cdot)$ is bounded and coercive, by the Lax-Milgram Theorem, one directly declares that there exists a unique weak solution $u \in V$ to the problem \eqref{eq:pde-weak}.

%
%

\subsection{Collocation discretization and preliminaries }\label{sub-2.3}

We begin by introducing some basic notation. Let \(T_{n}(x)\) denote the Chebyshev polynomials of the first kind. To construct basis functions that ensure the homogeneous Dirichlet boundary conditions in \eqref{eq:1d}, we define
\begin{equation}\label{eq:basis}
  \phi_{n}(x) = T_{n + 2}(x) - T_{n}(x), \qquad n = 0, 1, \ldots, N - 2.
\end{equation}

Note that \(T_{k}(1) = 1\) and \(T_{k}(-1) = (-1)^{k}\), hence

\[
\phi_{n}(1) = T_{n+2}(1) - T_{n}(1) = 0, \quad \phi_{n}(-1) = T_{n+2}(-1) - T_{n}(-1) = 0.
\]

Let

\[
V_{N} = \operatorname{span}\{\phi_{0}, \phi_{1}, \ldots, \phi_{N-2}\}.
\]

Take the approximate solution

\begin{equation}\label{eq:collo}
  u_{N}(x) = \sum_{n=0}^{N-2} a_{n} \phi_{n}(x).
\end{equation}

Then the corresponding discretized formula of \eqref{eq:1d} reads
\begin{equation}\label{eq:1d_distcr}
	\left\{ \begin{aligned}
		-\bigl(\omega(x) \, u_N'(x)\bigr)' & = f(x), \quad x \in I, \\[4pt]
		u_N(x)|_{\partial I} & = 0,
	\end{aligned}
	\right.
\end{equation}

Multiply both sides of \eqref{eq:1d_distcr} by a test function \(\phi_{m}(x)\) (\(m = 0, 1, \ldots, N-2\)) and integrate over \(I\), yielding

\[
\int_I \left( \sqrt{1 - x^{2}} \, u_{N}'(x) \right)' \phi_{m}(x) \, dx = \int_I f(x) \phi_{m}(x) \, dx, \quad m = 0, 1, \ldots, N-2.
\]

Substituting \eqref{eq:collo} and exchanging the order of summation and integration, we obtain

\[
\sum_{n=0}^{N-2} a_{n} \underbrace{\int_I \left( \sqrt{1 - x^{2}} \, \phi_{n}'(x) \right)' \phi_{m}(x) \, {\mathrm d}x}_{A_{nm}} = \int_I f(x) \phi_{m}(x) \, {\mathrm d}x, \quad m = 0, 1, \ldots, N-2.
\]

Therefore, the discrete system can be written in matrix form as
\[
\mathbf{A} \mathbf{a} = \mathbf{b},
\]
where ${\mathbf{A}}=(A_{nm})_{(N-1)\times(N-1)}$, $\mathbf{a} = (a_{0}, a_{1}, \ldots, a_{N-2})^{\mathrm T}$ and $\mathbf{b} = (b_{0}, b_{1}, \ldots, b_{N-2})^{\mathrm T}$.
Components of the right-hand side vector $\mathbf{b}$ are
\begin{equation}\label{eq:right}
b_{m} = \int_I f(x) \phi_{m}(x) \, {\mathrm d}x = \int_I f(x) \big( T_{m+2}(x) - T_{m}(x) \big) \, {\mathrm d}x.  
\end{equation}


Take in mind that, Chebyshev polynomials of the first kind satisfy
\begin{equation}\label{eq:cheb}
\left( \sqrt{1 - x^{2}} \, T_{k}'(x) \right)' = -\frac{k^{2}}{\sqrt{1 - x^{2}}} \, T_{k}(x), \qquad |x| < 1, \quad k \geq 0.
\end{equation}
Applying \eqref{eq:cheb} for \(k = n+2\) and \(k = n\), we get
\[
\left( \sqrt{1 - x^{2}} \, \phi_{n}'(x) \right)' = -\frac{(n+2)^{2}}{\sqrt{1 - x^{2}}} \, T_{n+2}(x) + \frac{n^{2}}{\sqrt{1 - x^{2}}} \, T_{n}(x).
\]
Since \(\phi_{m}(x) = T_{m+2}(x) - T_{m}(x)\), we have
\[
\begin{aligned}
	A_{nm} &= \int_I \left( \sqrt{1 - x^{2}} \, \phi_{n}'(x) \right)' \phi_{m}(x) \, {\mathrm d}x \\
	&= \int_i \frac{-(n+2)^{2} T_{n+2}(x) + n^{2} T_{n}(x)}{\sqrt{1 - x^{2}}} \big( T_{m+2}(x) - T_{m}(x) \big) \, dx \\
	&= -(n+2)^{2} \int_I \frac{T_{n+2}(x) T_{m+2}(x)}{\sqrt{1 - x^{2}}} \, dx + (n+2)^{2} \int_I \frac{T_{n+2}(x) T_{m}(x)}{\sqrt{1 - x^{2}}} \, dx \\
	&\quad\ + n^{2} \int_I \frac{T_{n}(x) T_{m+2}(x)}{\sqrt{1 - x^{2}}} \, dx - n^{2} \int_I \frac{T_{n}(x) T_{m}(x)}{\sqrt{1 - x^{2}}} \, dx.
\end{aligned}
\]

Chebyshev polynomials of the first kind are orthogonal with respect to the weight function \(w(x) = (1 - x^{2})^{-1/2}\):
\[
\int_I \frac{T_{n}(x) T_{m}(x)}{\sqrt{1 - x^{2}}} \, {\mathrm d}x = \left\{ 
\begin{array}{ll}
	0, & n \neq m, \\
	\pi, & n = m = 0, \\
	\frac{\pi}{2}, & n = m \geq 1.
\end{array} 
\right.
\]

Denote
\[
h_{n} = \int_I \frac{T_{n}(x)^{2}}{\sqrt{1 - x^{2}}} \, {\mathrm d}x = \left\{ 
\begin{array}{ll}
	\pi, & n = 0, \\
	\frac{\pi}{2}, & n \geq 1.
\end{array} 
\right.
\]
Thus the entries of \(A_{mn}\) are nonzero only when \(m = n\) or \(m = n \pm 2\), and can be given by explicit formulae as
\[
A_{mn} = \left\{ 
\begin{array}{ll}
	-(n+2)^2 h_{n+2} - n^2 h_n, & m = n \\
	\displaystyle (n+2)^2 h_{n+2}, & m = n + 2 \\
	\displaystyle n^2 h_n, & m = n - 2 \\
	0, & \text{otherwise}
\end{array} 
\right. \quad (m, n = 0, 1, \ldots, N-2) 
\]

To facilitate numerical computation, we introduce a variable substitution
\[
x = \cos \theta, \qquad \theta \in [0, \pi], \qquad {\mathrm d}x = -\sin \theta \, {\mathrm d}\theta,
\]
hence the original expression in \eqref{eq:right} can be equivalently written as
\[
b_{m} = \int_{0}^{\pi} f(\cos \theta) \big( \cos((m+2)\theta) - \cos(m\theta) \big) \sin \theta \, {\mathrm d}\theta,
\]
where the first kind Chebyshev polynomials
\[
T_{k}(\cos \theta) = \cos(k\theta).
\]

We use \(Q\)-point Gauss-Legendre quadrature on \([0, \pi]\) with a linear mapping to calculate 
\[
I_{m} = \int_I f(x) T_{m}(x) \, {\mathrm d}x = \int_{0}^{\pi} f(\cos \theta) \cos(m\theta) \sin \theta \, {\mathrm d}\theta. 
\]
Then the original expression within \eqref{eq:right} equals
\begin{equation}\label{eq:right-eq}
b_{m} = I_{m+2} - I_{m}, \quad m = 0, 1, \ldots, N-2.
\end{equation}
Compute using Gauss-Legendre quadrature, $i.e.$,
\[
I_{m} \approx I_{m}^{(Q)} = \sum_{j=1}^{Q} \omega_{j} f(\cos \theta_{j}) \cos(m\theta_{j}) \sin \theta_{j},
\]
then one readily obtains \(b_{m} \approx I_{m+2}^{(Q)} - I_{m}^{(Q)}\), which can be readily used within the numerical calculations.


\section{A Priori Error Estimate}

In order to investigate the error estimates for the proposed collocation schemes, a gradient orthogonal projector is introduced. For any $v\in V$, it is defined with respect to
 $\mathfrak{P}_{N}: V\mapsto \tilde{V}_N$ as:
\[
(\nabla (v - \mathfrak{P}_{N} v), \nabla w_{N})_{\omega} = 0, \quad \forall w_{N} \in \tilde{V}_{N},
\]
where
\[
\tilde{V}_N = \operatorname{span}\{\phi_{0}(x), \phi_{1}(x), \ldots, \phi_{N-2}(x), \phi_{N-1}(x), \phi_{N}(x)\}.
\]

Moreover, the projection operator \(\mathfrak{P}_{N}\) is the Ritz projection (elliptic projection) in \(V\), satisfying the following stability:
\[
\| \mathfrak{P}_{N} v \|_{1} \leq C \| v \|_{1}, \quad \forall v \in V.
\]

For a given \(F \in H^{-1}(I)\), define \(y_{F} \in V\) as the corresponding weak solution of the following problem:
\[
a( y_{F}, v) = \langle F, v \rangle, \quad \forall v \in V.
\]

Therefore,

\[
\| y_{F} - \mathfrak{P}_{N} y_{F} \|_{V} \leq \| y_{F} \|_{V} + \| \mathfrak{P}_{N} y_{F} \|_{V} \leq (1 + C) \| y_{F} \|_{V} \leq c \| F \|_{-1}.
\]

%

It is clear that 
\[
\left| a(u - \mathfrak{P}_{N} u, y_{F} - \mathfrak{P}_{N} y_{F}) \right| \leq \| u -\mathfrak{P}_{N} u \|_{V} \|  y_{F} - \mathfrak{P}_{N} y_{F}\|_{V}.
\]

%
%
%

Then, we obtain
\[
\sup_{\forall F \in H^{-1}(\Omega)} \frac{a(u - \mathfrak{P}_N u, y_{F} - \mathfrak{P}_N y_{F})}{\| F \|_{-1}} \leq \| u - \mathfrak{P}_N u \|_{V} \cdot \sup_{\forall F \in H^{-1}(\Omega)} \frac{\| y_{F} - \mathfrak{P}_N y_{F} \|_{V}}{\| F \|_{-1}}.
\]



Assuming the exact solution \(u \in H^{m}(I)\cap H_{0,\omega}^1(I)\), and using polynomial interpolation approximations, one gets that the projection operator error satisfies:

\[
\| u - \mathfrak{P}_N u \|_{1} \leq C N^{1-m} \| u \|_{m}.
\]

In particular, there holds:

\begin{equation}\label{u_est}
\| u - \mathfrak{P}_N u \|_{V} \leq c \| u - \mathfrak{P}_N u \|_{1} \leq C N^{1-m} \| u \|_{m}.
\end{equation}

Now we are at the point to discuss the error estimates of our proposed collocation approximations. 
\begin{theorem}
	Assume $u\in H^m(I)\cap H_{0,\omega}^1$	with $m\geq 2$.
	Let $u$ and $u_{N}$ be the solution of \eqref{eq:1d} and \eqref{eq:1d_distcr}, respectively. There holds the following a-priori error estimates: 
	\begin{equation}\label{a_pri_1}
	   \begin{aligned}
		\| u - u_{N} \|_{V} \leq C N^{1-m} \| u \|_{m},
    	\end{aligned}
	\end{equation}
	and
		\begin{equation}\label{a_pri_0}
		\begin{aligned}
			\| u - u_{N} \|_{U} \leq C N^{-m} \| u \|_{m}.
		\end{aligned}
	\end{equation}
\end{theorem}

\begin{proof}
We present the details as follows.
\[
\begin{aligned}
	\| u - u_{N} \|_{V} &= \sup_{\forall F \in H^{-1}(I)} \frac{|\langle F, u - u_{N} \rangle|}{\| F \|_{-1}} 
	= \sup_{\forall F \in H^{-1}(I)} \frac{a (u - u_{N}, y_{F})}{\| F \|_{-1}}\\
	&= \sup_{\forall F \in H^{-1}(I)} \frac{a(u - \mathfrak{P}_N u, y_{F} - \mathfrak{P}_N y_{F})}{\| F \|_{-1}} \\
	&\leq c \| u - \mathfrak{P}_N u \|_{V} \leq C N^{1-m} \| u \|_{m}.
\end{aligned}
\]
where we used the result stated in \eqref{u_est}.

In the light of Riesz theorem, we investigate the error estimates of collocation approximations with $L^2$-norm: 
\[
\begin{aligned}
	& \quad \| u - u_{N} \|_U = \sup_{\forall G \in L^{2}(\Omega)} \frac{|\langle G, u - u_{N} \rangle|}{\| G \|} = \sup_{\forall G \in L^{2}(\Omega)} \frac{|a(y_{G}, u - u_{N})|}{\| G \|} \\
	& = \sup_{\forall G \in L^{2}(\Omega)} \frac{|a(y_{G} - y_{G}^{N}, u - u_{N})|}{\| G \|} 
	\leq C \| u - u_{N} \|_{V} \sup_{\forall G \in L^{2}(\Omega)} \frac{\| \omega(x) \nabla (y_{G} - y_{G}^{N}) \|}{\| G \|} \\
	&\leq c \| u - u_{N} \|_{V} \sup_{\forall G \in L^{2}(\Omega)} \frac{\| \nabla (y_{G} - y_{G}^{N}) \|}{\| G \|} \leq C \| u - u_{N} \|_{V} \sup_{\forall G \in L^{2}(\Omega)}  \frac{\| y_{G} - \mathfrak{P}_N y_{G} \|_{1}}{\| G \|} \\
	&\leq cN^{-1} \| u - u_{N} \|_{V} \cdot \frac{\| y_{G} \|_{2}}{\| G \|} \leq C N^{-m} \| u \|_{m}. 
\end{aligned}
\]
This is the desired results.
\end{proof}

\section{Numerical Experiments}

In this section, we perform some numerical experiments to confirm
the analytical results given in above sections. We use our proposed schemes to calculate the model with some different given solutions, which depict the spectral convergence of numerical approximations.

\textbf{Example 1.} To assess the convergence of the proposed collocation schemes, we select an analytical solution that possesses limited regularity specifically, below first order:
\[
  u = -\sqrt{1 - x^{2}}.
\]
And the right hand item reads
\[
  f = 1.
\]

The table below presents, for different values of \(N\), the numerical data of the collocation method: the error between the numerical and exact solutions in both the \(L^{\infty}\)- and \(L^{2}\)-norms, the corresponding convergence order for the solution $u$, and the condition number of the stiffness matrix $A$. And the rate is defined by
\[
\mathrm{rate} = \frac{\log(E(N_2) / E(N_1))}{\log (N_1/N_2)}, \quad E(N) = \| u - u_{N} \|_{L^2}.
\]

\begin{table}[H]
	\centering
		\caption{Errors, convergence rates and condition numbers for different $N$.}
	\begin{tabular}{|c|c|c|c|c|c|c|}
		\hline
		\(N\) & \(\|u-u_{N}\|_{L^\infty}\) & \(\|u'-u'_{N}\|_{L^\infty}\) & \(\|u-u_{N}\|_{L^2}\) & \(\|u'-u'_{N}\|_{L^2}\) & rate & cond(\(A\)) \\
		\hline
		4 & 1.7357e-01 & 7.0711e+05 & 1.8518e-01 & 2.2946 & / & 6.8541e+00 \\
		8 & 9.5699e-02 & 7.0711e+05 & 1.0111e-01 & 2.0221 & 1.823e-01 & 4.8501e+01 \\
		12 & 6.6134e-02 & 7.0711e+05 & 6.9649e-02 & 1.8404 & 2.181e-01 & 1.4231e+02 \\
		16 & 5.0536e-02 & 7.0710e+05 & 5.3149e-02 & 1.7001 & 2.502e-01 & 2.9647e+02 \\
		24 & 3.4347e-02 & 7.0710e+05 & 3.6078e-02 & 1.5882 & 2.126e-01 & 8.0452e+02 \\
		\hline
	\end{tabular}
	\label{tab:example1}
\end{table}

Table \ref{tab:example1} summarizes the numerical results for various polynomial degrees $N$. A key observation is the behavior of the derivative errors. $\|u'-u'_N\|_{L^\infty}$ remains consistently large ($\approx 7.07 \times 10^5$) across all $N$. This is not a failure of the method but a direct consequence of the boundary singularity of the exact derivative, which becomes unbounded at the endpoints.

 \begin{figure}[H]
	 \centering
	 	\begin{subfigure}[t]{0.5\textwidth}
	 	\centering
	 	\includegraphics[width=\linewidth]{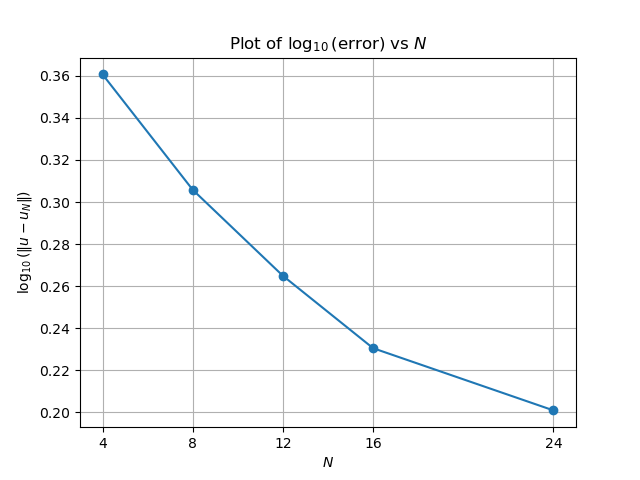}
	 	\caption{\(\log_{10}(\| u - u_{N}\|_{L^2})\) versus \(N\).}
	 \end{subfigure}\hfill
	 \begin{subfigure}[t]{0.5\textwidth}
	 \includegraphics[width=8.5cm, height=6.2cm]{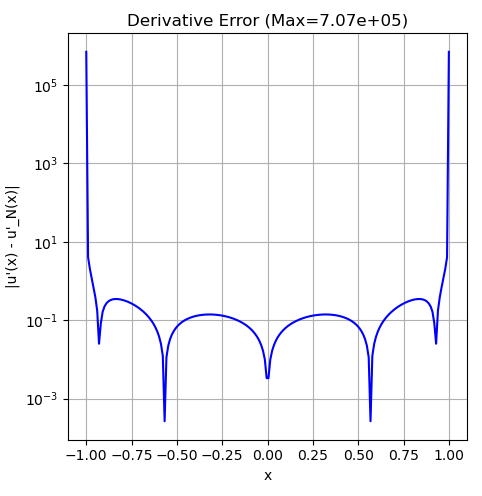}
	 \caption{\(|u'-u'_{N}|\) with \(N=4\).}
	\end{subfigure}
	\caption{Solution convergence profile and pointwise error for $N=4$.}
	\label{fig:example1_convergence}
	 \end{figure}
\begin{figure}[H]
	\centering
	
	\begin{subfigure}[t]{0.5\textwidth}
		\centering
		\includegraphics[width=\linewidth]{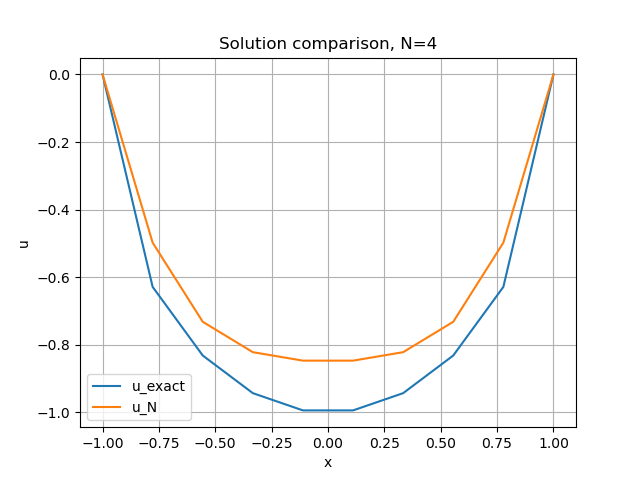}
		\caption{Sketches of $u$ and $u_N$}
	\end{subfigure}\hfill
	\begin{subfigure}[t]{0.5\textwidth}
		\centering
		\includegraphics[width=\linewidth]{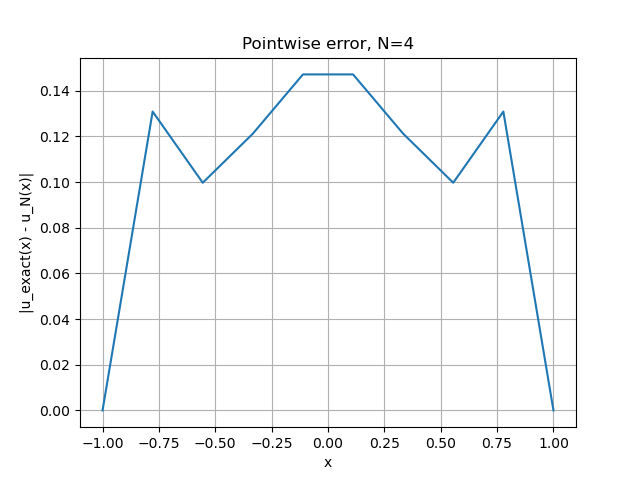}
		\caption{Pointwise error of $|u-u_N|$}
	\end{subfigure}
	\caption{Solution profile and pointwise error for $N=4$. (left: true solution $u$ versus numerical solution $u_N$. right: $|u-u_N|$)}
\label{fig:example1_sol_err1}
\end{figure}

\begin{figure}[H]
\centering
	
	\begin{subfigure}[t]{0.5\textwidth}
		\centering
		\includegraphics[width=\linewidth]{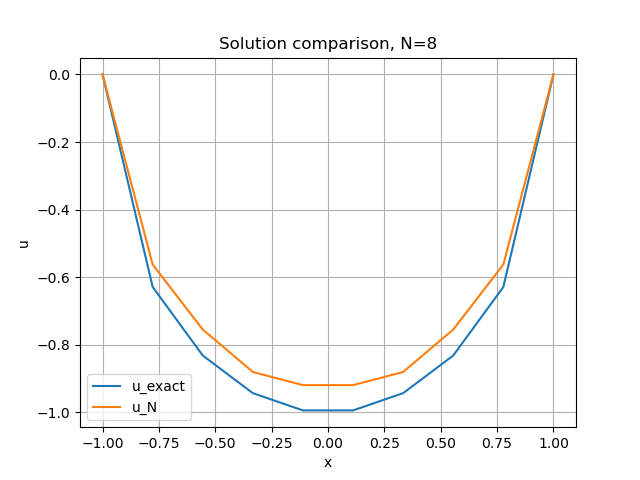}
		\caption{Sketches of $u$ and $u_N$ }
	\end{subfigure}\hfill
	\begin{subfigure}[t]{0.5\textwidth}
		\centering
		\includegraphics[width=\linewidth]{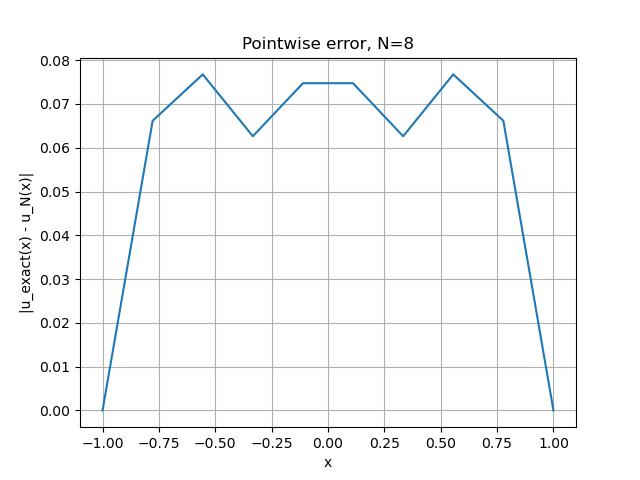}
		\caption{Pointwise error of $|u-u_N|$ }
	\end{subfigure}
	\caption{Solution profile and pointwise error for $N=8$. (left: true solution $u$ versus numerical solution $u_N$. right: $|u-u_N|$)}
\label{fig:example1_sol_err2}
\end{figure}

\begin{figure}[H]
\centering
	
	\begin{subfigure}[t]{0.5\textwidth}
		\centering
		\includegraphics[width=\linewidth]{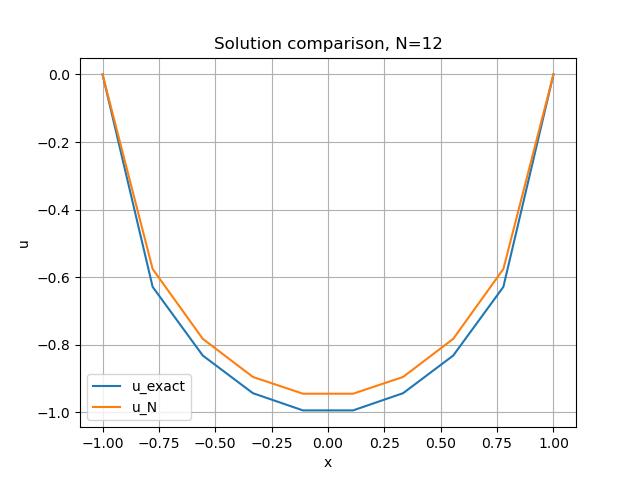}
		\caption{Sketches of $u$ and $u_N$ }
	\end{subfigure}\hfill
	\begin{subfigure}[t]{0.5\textwidth}
		\centering
		\includegraphics[width=\linewidth]{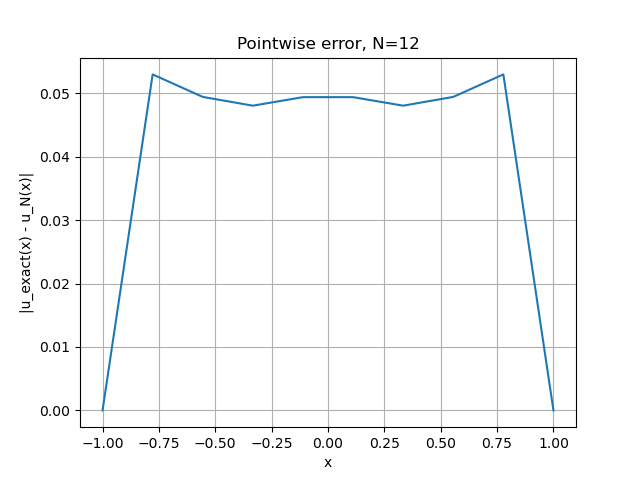}
		\caption{Pointwise error of $|u-u_N|$ }
	\end{subfigure}
	
	\caption{Solution profile and pointwise error for $N=12$. (left: true solution $u$ versus numerical solution $u_N$. right: $|u-u_N|$)}
\label{fig:example1_sol_err3}
\end{figure}

\begin{figure}[H]
\centering
	
	\begin{subfigure}[t]{0.5\textwidth}
		\centering
		\includegraphics[width=\linewidth]{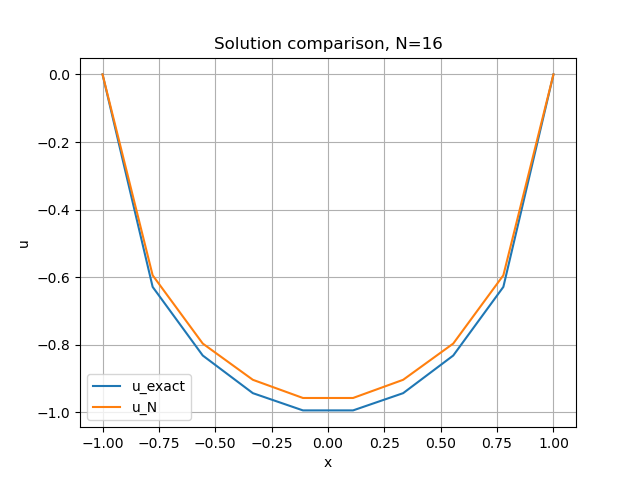}
		\caption{Sketches of $u$ and $u_N$ }
	\end{subfigure}\hfill
	\begin{subfigure}[t]{0.5\textwidth}
		\centering
		\includegraphics[width=\linewidth]{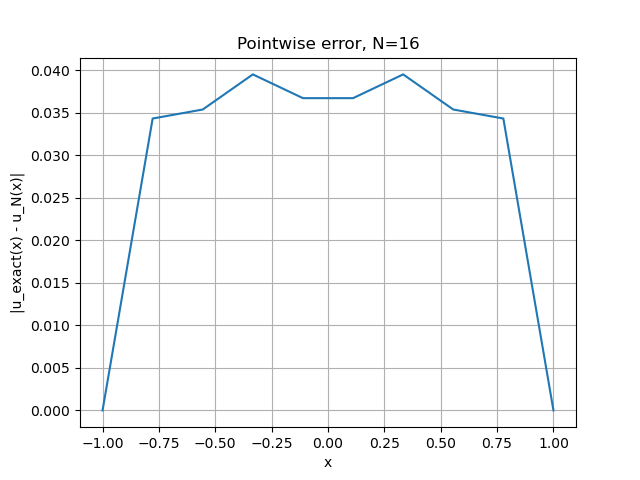}
		\caption{Pointwise error of $|u-u_N|$ }
	\end{subfigure}
	
	\caption{Solution profile and pointwise error for $N=16$. (left: true solution $u$ vs. numerical solution $u_N$. right: $|u-u_N|$)}
	\label{fig:example1_sol_err4}
\end{figure}

Figures \ref{fig:example1_convergence}-\ref{fig:example1_sol_err4} reveal that the approximation quality is relatively poor, with a visible discrepancy between the numerical and exact solutions even at $N=16$. This aligns with Table~\ref{tab:example1}, where the extremely large $L^\infty$ derivative errors and slow convergence rate ($\approx 0.2$) confirm that the boundary singularity severely limits the method's global accuracy (as in Figure \ref{fig:example1_convergence} (b)).

\textbf{Example 2.} To compare the efficiency of the proposed collocation schemes, we select an analytical solution that possesses only limited regularity, specifically, of order approximately one:
\[
  u = -(1 - x^{2})^{3/2}.
\]
And the corresponding right hand item is
\[
  f = 3 - 9x^{2},
\]

\begin{table}[H]
	\centering
	\caption{Errors, convergence rates and condition numbers for different $N$.}
	\begin{tabular}{|c|c|c|c|c|c|c|}
		\hline
		\(N\) & \(\|u-u_{N}\|_{L^\infty}\) & \(\|u'-u'_{N}\|_{L^\infty}\) & \(\|u-u_{N}\|_{L^2}\) & \(\|u'-u'_{N}\|_{L^2}\) & rate & cond(\(A\)) \\
		\hline
		4 & 1.8482e-02 & 3.4144e-02 & 2.0551e-02 & 1.3945e-01 & / & 6.8541e+00 \\
		8 & 2.4235e-03 & 7.0750e-03 & 2.7266e-03 & 3.5736e-02 & 1.9643 & 4.8501e+01 \\
		12 & 7.4183e-04 & 2.0189e-03 & 8.6331e-04 & 1.6528e-02 & 1.9010 & 1.4231e+02 \\
		16 & 3.1152e-04 & 9.7438e-04 & 3.7859e-04 & 9.5296e-03 & 1.9060 & 2.9647e+02 \\
		24 & 9.3095e-05 & 2.9933e-04 & 1.1714e-04 & 4.3571e-03 & 1.9235 & 8.0452e+02 \\
		\hline
	\end{tabular}
	
	\label{tab:example2}
\end{table}

Table \ref{tab:example2} presents the quantitative errors and convergence rates for Example 2. In contrast to the singular behaviors observed in Example 1, the derivative errors here are well-behaved and decrease monotonically with $N$. The convergence rate stabilizes at approximately $1.9$--$2.0$, indicating a stable algebraic convergence of second order. Although the improved regularity of the solution $u(x) = -(1-x^2)^{3/2}$ leads to higher accuracy compared to Example 1, the convergence remains algebraic, confirming that finite regularity continues to preclude spectral convergence.

\begin{figure}[H]
	\centering
	\includegraphics[width=0.5\textwidth]{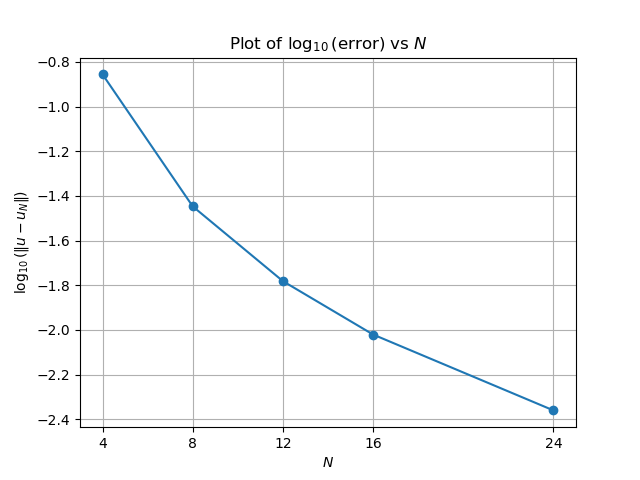}
	\caption{\(\log_{10}(\| u - u_{N}\|_{L^2})\) versus \(N\).}
	\label{fig:example2_convergence}
\end{figure}

\begin{figure}[H]
	\centering
	
	\begin{subfigure}[t]{0.5\textwidth}
		\centering
		\includegraphics[width=\linewidth]{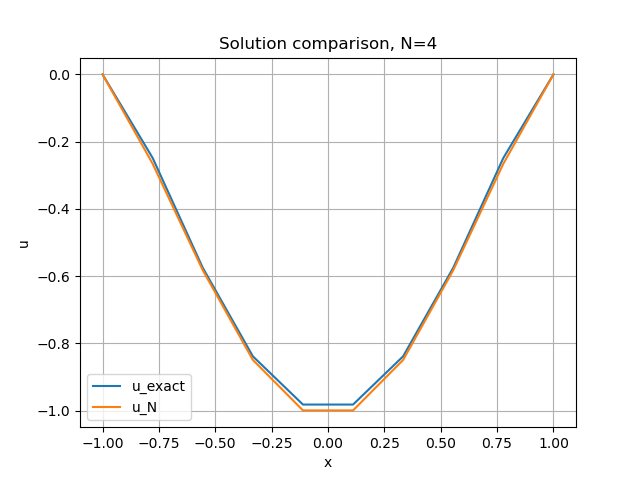}
		\caption{Sketches of $u$ and $u_N$ }
	\end{subfigure}\hfill
	\begin{subfigure}[t]{0.5\textwidth}
		\centering
		\includegraphics[width=\linewidth]{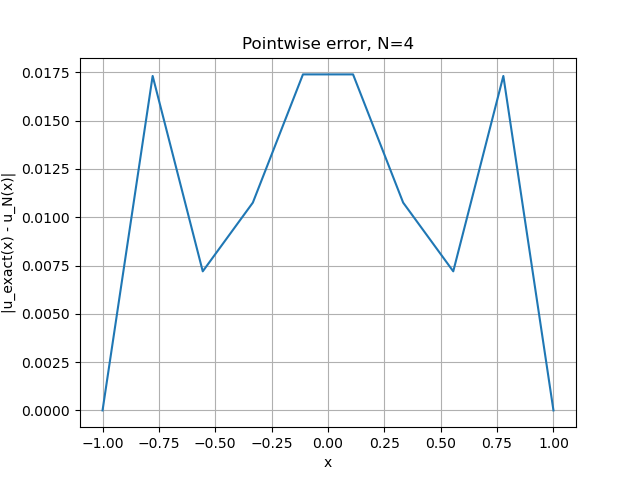}
		\caption{Pointwise error of $|u-u_N|$ }
	\end{subfigure}
	
	\caption{Solution profile and pointwise error for $N=4$. (left: true solution $u$ versus numerical solution $u_N$. right: $|u-u_N|$)}
	\label{fig:example2_sol_err1}
\end{figure}

\begin{figure}[H]
	\centering

	\begin{subfigure}[t]{0.5\textwidth}
		\centering
		\includegraphics[width=\linewidth]{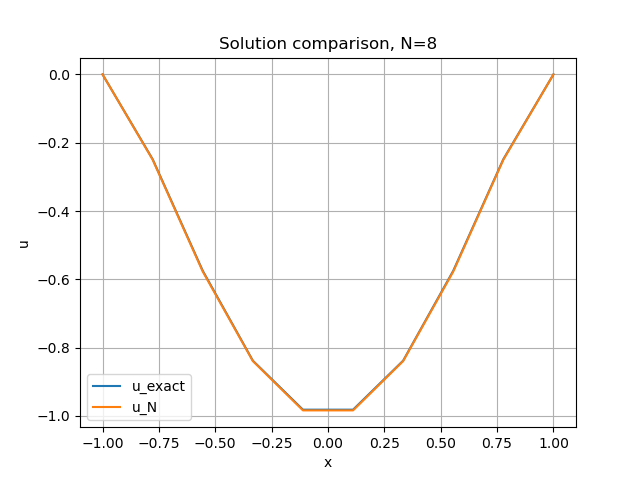}
		\caption{Sketches of $u$ and $u_N$ }
	\end{subfigure}\hfill
	\begin{subfigure}[t]{0.5\textwidth}
		\centering
		\includegraphics[width=\linewidth]{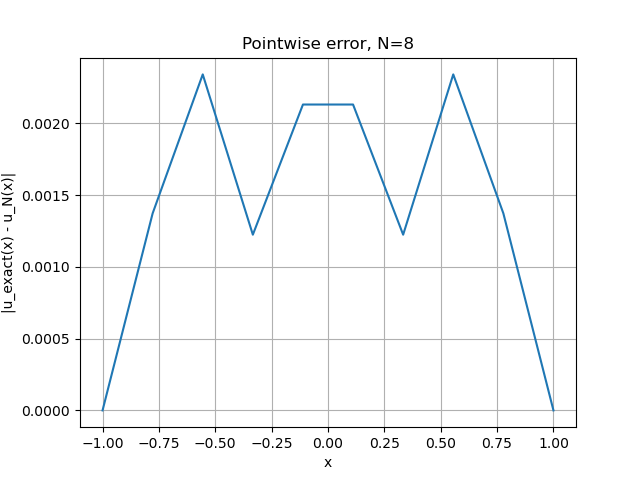}
		\caption{Pointwise error of $|u-u_N|$ }
	\end{subfigure}
	
	\caption{Solution profile and pointwise error for $N=8$. (left: true solution $u$ versus numerical solution $u_N$. right: $|u-u_N|$)}
	\label{fig:example2_sol_err2}
\end{figure}

\begin{figure}[H]
	\centering

	\begin{subfigure}[t]{0.5\textwidth}
		\centering
		\includegraphics[width=\linewidth]{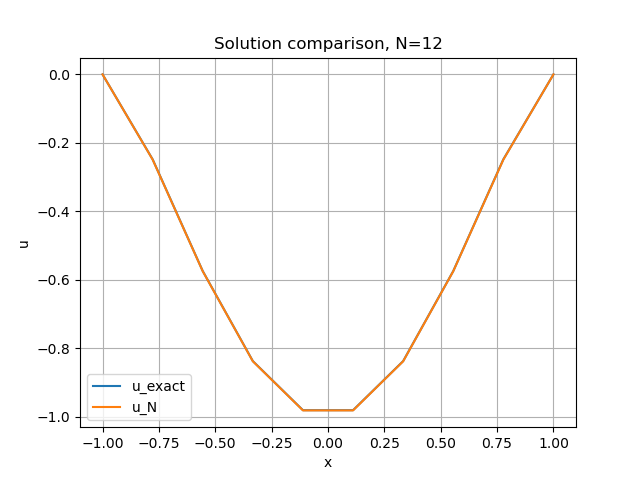}
		\caption{Sketches of $u$ and $u_N$ }
	\end{subfigure}\hfill
	\begin{subfigure}[t]{0.5\textwidth}
		\centering
		\includegraphics[width=\linewidth]{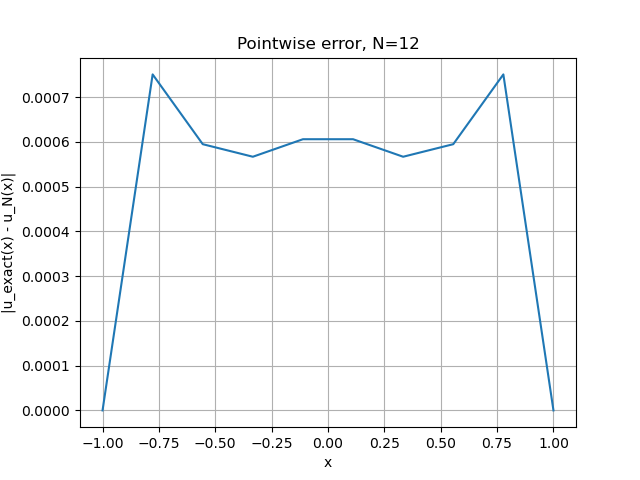}
		\caption{Pointwise error of $|u-u_N|$ }
	\end{subfigure}
	
	\caption{Solution profile and pointwise error for $N=12$. (left: true solution $u$ versus numerical solution $u_N$. right: $|u-u_N|$)}
	\label{fig:example2_sol_err3}
\end{figure}

\begin{figure}[H]
	\centering

	\begin{subfigure}[t]{0.5\textwidth}
		\centering
		\includegraphics[width=\linewidth]{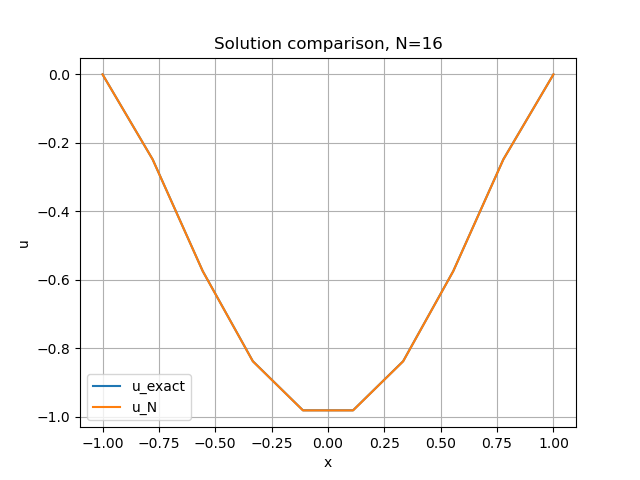}
		\caption{Sketches of $u$ and $u_N$ }
	\end{subfigure}\hfill
	\begin{subfigure}[t]{0.5\textwidth}
		\centering
		\includegraphics[width=\linewidth]{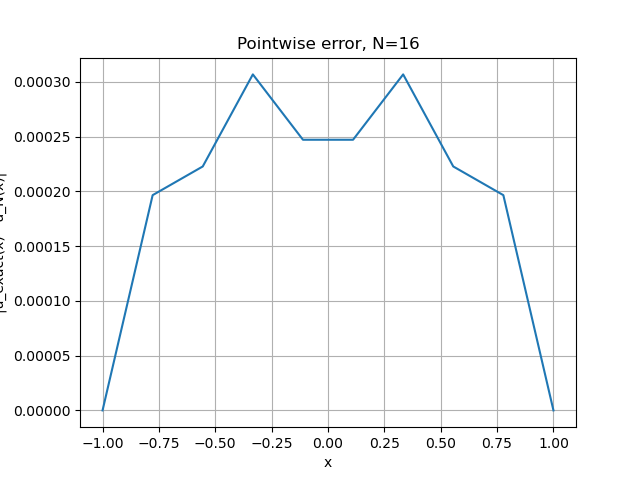}
		\caption{Pointwise error of $|u-u_N|$ }
	\end{subfigure}
	
	\caption{Solution profile and pointwise error for $N=16$. (left: true solution $u$ versus numerical solution $u_N$. right: $|u-u_N|$)}
	\label{fig:example2_sol_err4}
\end{figure}

Figure \ref{fig:example2_convergence}--\ref{fig:example2_sol_err4} displays the error decay trend, which follows a steady algebraic slope consistent with the results in Table~\ref{tab:example2}. The pointwise comparisons in Figures~\ref{fig:example2_sol_err1}--\ref{fig:example2_sol_err4} demonstrate a significant improvement in approximation quality compared to Example 1. The numerical solution $u_N$ closely matches the exact solution $u$ even for small $N$. Similar to the previous case, the pointwise errors vanish at the boundaries and exhibit a wave-like distribution within the domain, but with magnitudes reduced by approximately two orders compared to Example 1.

\textbf{Example 3.} A smooth solution is chosen to illustrate the convergence behavior of the proposed collocation schemes:
\[
  u = (1 - x^{2}) \sin(\pi x).
\]
And the corresponding right hand item can be calculated as
\[
 f = \frac{(4x^{2} - 2) - \pi^{2}(1 - x^{2})^{2}}{\sqrt{1 - x^{2}}} \sin(\pi x) - 5\pi x \sqrt{1 - x^{2}} \cos(\pi x),
\]

\begin{table}[H]
	\centering
		\caption{Errors, convergence rates and condition numbers for different $N$.}
	\begin{tabular}{|c|c|c|c|c|c|c|}
		\hline
		\(N\) & \(\|u-u_{N}\|_{L^\infty}\) & \(\|u'-u'_{N}\|_{L^\infty}\) & \(\|u-u_{N}\|_{L^2}\) & \(\|u'-u'_{N}\|_{L^2}\) & rate & cond(\(A\)) \\
		\hline
		4 & 4.5438e-01 & 4.6092e+00 & 3.4885e-01 & 2.0836e+00 & / & 6.8541e+00 \\
		8 & 3.3160e-03 & 1.3665e-01 & 2.2368e-03 & 3.2578e-02 & 5.9990 & 4.8501e+01 \\
		12 & 2.8161e-06 & 2.4333e-04 & 1.8090e-06 & 4.0908e-05 & 13.2262 & 1.4231e+02 \\
		16 & 5.5465e-10 & 8.1694e-08 & 3.4705e-10 & 1.0711e-08 & 21.5364 & 2.9647e+02 \\
		24 & 4.2188e-15 & 2.8156e-10 & 3.0755e-15 & 1.0553e-14 & 31.8520 & 8.0452e+02 \\
		\hline
	\end{tabular}

	\label{tab:example3}
\end{table}

Table~\ref{tab:example3} reveals a dramatic change in convergence behaviors compared to the previous non-smooth examples. For a smooth solution $u(x) = (1-x^2)\sin(\pi x)$, the errors decay rapidly. Notably, as $N$ increases from 4 to 24, the $L^2$-error drops from $O(10^{-1})$ to $O(10^{-15})$, essentially reaching machine precision.
Furthermore, the computed convergence rates are not constant but increase significantly with $N$ (from $\approx 6.0$ to $\approx 31.9$). This increasing rate is a hallmark of spectral convergence (or exponential convergence), confirming that the proposed Chebyshev collocation method is highly efficient for smooth problems.

 \begin{figure}[H]
	 \centering
	 \includegraphics[width=0.5\textwidth]{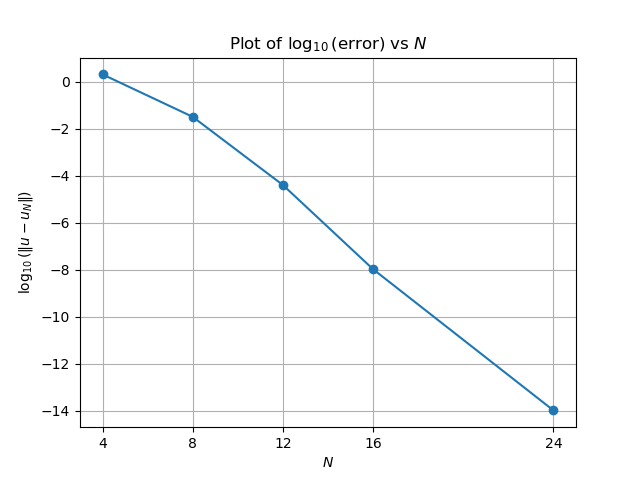}
	 \caption{\(\log_{10}(\| u - u_{N}\|_{L^2})\) versus \(N\).}
	 \label{fig:example3_convergence}
	 \end{figure}

\begin{figure}[H]
	\centering
	
	\begin{subfigure}[t]{0.5\textwidth}
		\centering
		\includegraphics[width=\linewidth]{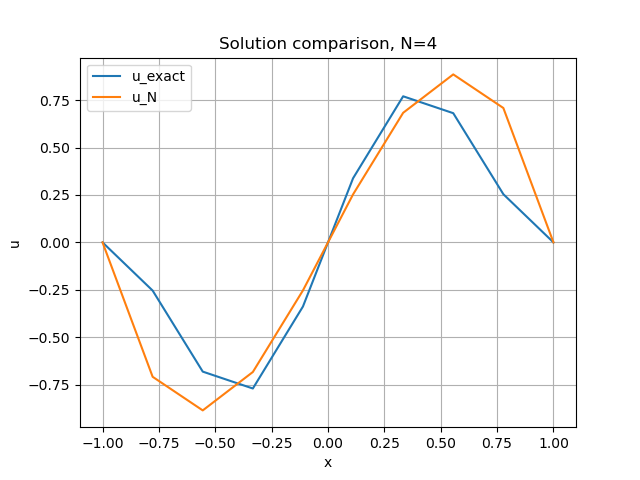}
		\caption{Sketches of $u$ and $u_N$ }
	\end{subfigure}\hfill
	\begin{subfigure}[t]{0.5\textwidth}
		\centering
		\includegraphics[width=\linewidth]{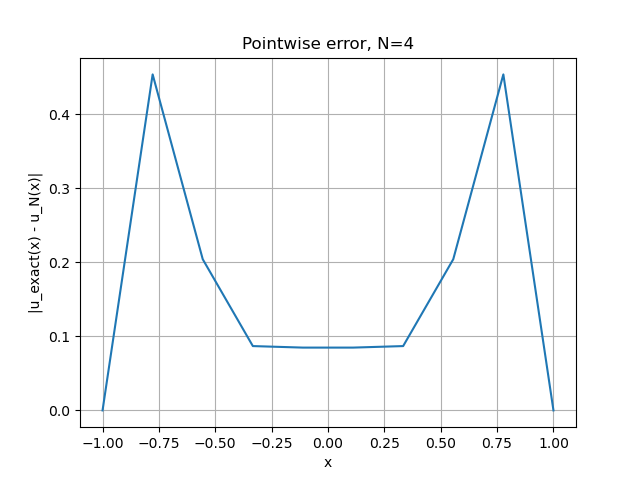}
		\caption{Pointwise error of $|u-u_N|$ }
	\end{subfigure}
	
	\caption{Solution profile and pointwise error for $N=4$. (left: true solution $u$ versus numerical solution $u_N$. right: $|u-u_N|$)}
\label{fig:example3_sol_err1}
\end{figure}

\begin{figure}[H]
\centering
	
	\begin{subfigure}[t]{0.5\textwidth}
		\centering
		\includegraphics[width=\linewidth]{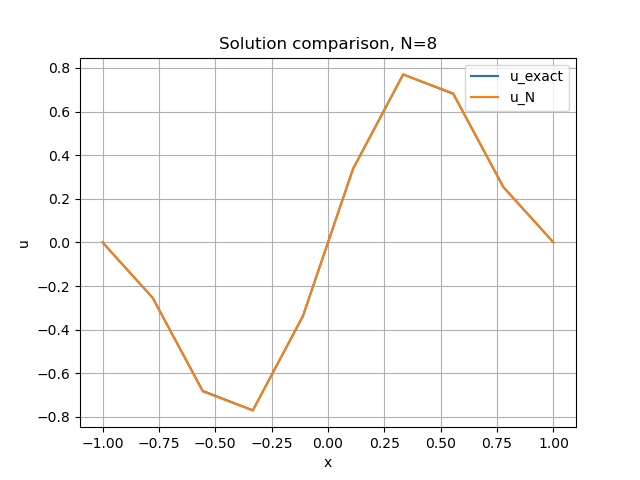}
		\caption{Sketches of $u$ and $u_N$ }
	\end{subfigure}\hfill
	\begin{subfigure}[t]{0.5\textwidth}
		\centering
		\includegraphics[width=\linewidth]{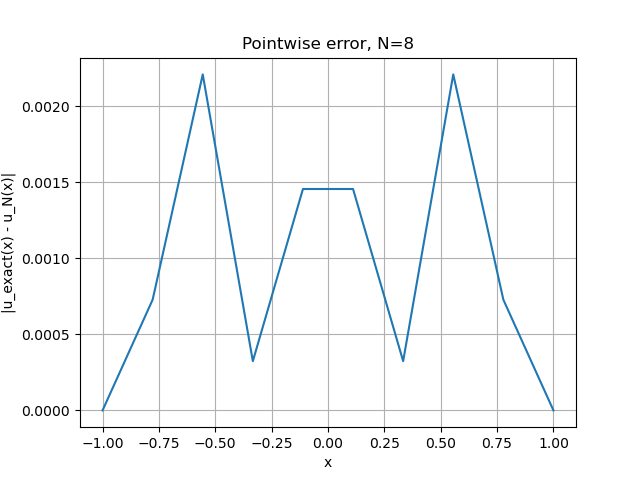}
		\caption{Pointwise error of $|u-u_N|$ }
	\end{subfigure}
	
	\caption{Solution profile and pointwise error for $N=8$. (left: true solution $u$ versus numerical solution $u_N$. right: $|u-u_N|$)}
\label{fig:example3_sol_err2}
\end{figure}

\begin{figure}[H]
\centering
	
	\begin{subfigure}[t]{0.5\textwidth}
		\centering
		\includegraphics[width=\linewidth]{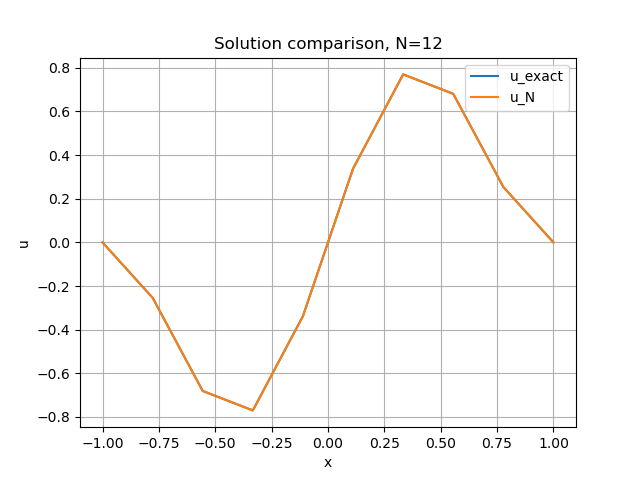}
		\caption{Sketches of $u$ and $u_N$ }
	\end{subfigure}\hfill
	\begin{subfigure}[t]{0.5\textwidth}
		\centering
		\includegraphics[width=\linewidth]{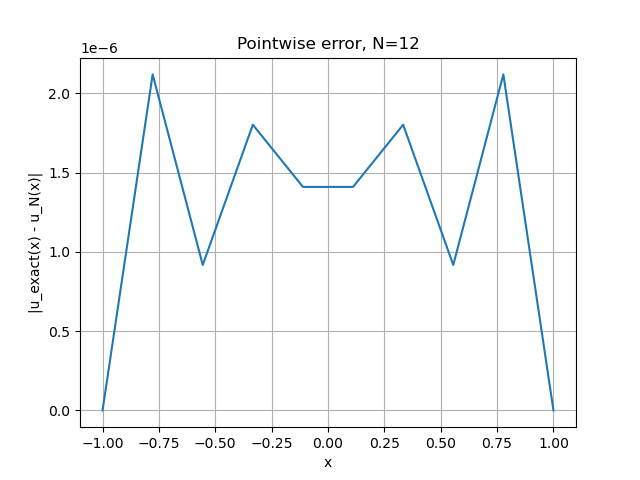}
		\caption{Pointwise error of $|u-u_N|$ }
	\end{subfigure}
	
	\caption{Solution profile and pointwise error for $N=12$. (left: true solution $u$ versus numerical solution $u_N$. right: $|u-u_N|$)}
\label{fig:example3_sol_err3}
\end{figure}

\begin{figure}[H]
\centering
	
	\begin{subfigure}[t]{0.5\textwidth}
		\centering
		\includegraphics[width=\linewidth]{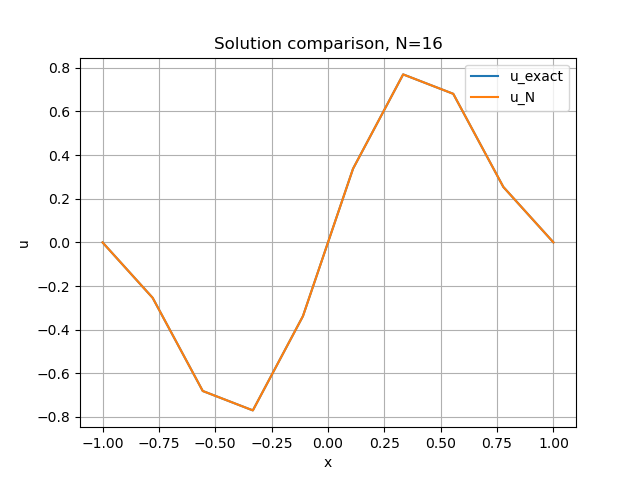}
		\caption{Sketches of $u$ and $u_N$ }
	\end{subfigure}\hfill
	\begin{subfigure}[t]{0.5\textwidth}
		\centering
		\includegraphics[width=\linewidth]{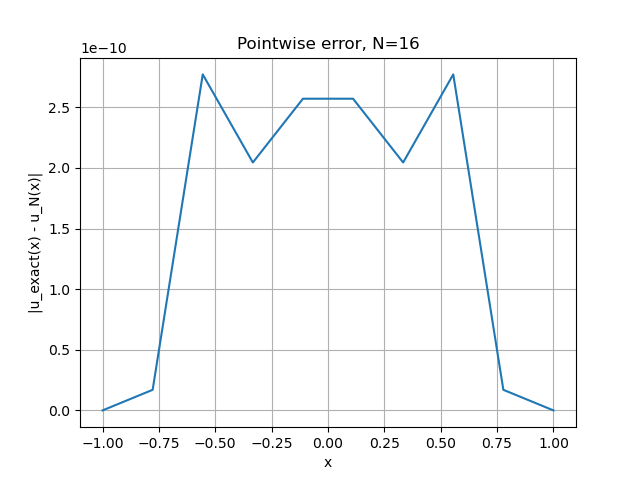}
		\caption{Pointwise error of $|u-u_N|$ }
	\end{subfigure}
	
	\caption{Solution profile and pointwise error for $N=16$. (left: true solution $u$ versus  numerical solution $u_N$. right: $|u-u_N|$)}
	\label{fig:example3_sol_err4}
\end{figure}

Figure \ref{fig:example3_convergence} presents the semi-logarithmic plot of the $L^2$-error versus $N$. The curve exhibits a clear linear decay, which visually confirms the exponential convergence property predicted by spectral theory.
Figures~\ref{fig:example3_sol_err1}--\ref{fig:example3_sol_err4} further illustrate this rapid convergence. While a slight discrepancy is visible at $N=4$ (Figure~\ref{fig:example3_sol_err1}), the numerical solution $u_N$ becomes indistinguishable from the exact solution $u$ for $N \ge 8$. The pointwise error plots (right panels) show a dramatic reduction in magnitude as $N$ increases, dropping from $O(10^{-1})$ at $N=4$ to $O(10^{-10})$ at $N=16$, effectively capturing the smooth solution with high precision.

\textbf{Example 4.} In this numerical experiment, we consider an exact solution that is sufficiently smooth:
\[
u = e^{x^{2} - 1} - 1,
\]
and the smooth corresponding right hand item:
\[
f = \frac{2 e^{x^{2} - 1} (1 - 2x^{4})}{\sqrt{1 - x^{2}}}.
\]

\begin{table}[H]
	\centering
	\caption{Errors, convergence rates and condition numbers for different $N$.}
	\begin{tabular}{|c|c|c|c|c|c|c|}
		\hline
		\(N\) & \(\|u-u_{N}\|_{L^\infty}\) & \(\|u'-u'_{N}\|_{L^\infty}\) & \(\|u-u_{N}\|_{L^2}\) & \(\|u'-u'_{N}\|_{L^2}\) & rate & cond(\(A\)) \\
		\hline
		4 & 6.5962e-03 & 1.2441e-01 & 5.6563e-03 & 3.8102e-02 & / & 6.8541e+00 \\
		8 & 2.0828e-05 & 1.0639e-03 & 1.7878e-05 & 2.1259e-04 & 7.4856 & 4.8501e+01 \\
		12 & 3.0506e-08 & 3.0241e-06 & 2.6194e-08 & 4.5429e-07 & 12.8474 & 1.4231e+02 \\
		16 & 2.6299e-11 & 4.2894e-09 & 2.2591e-11 & 5.1765e-10 & 18.6477 & 2.9647e+02 \\
		24 & 2.9976e-15 & 2.0650e-14 & 2.9345e-15 & 5.4655e-15 & 26.3087 & 8.0452e+02 \\
		\hline
	\end{tabular}
	
	\label{tab:example4}
\end{table}

Table~\ref{tab:example4} corroborates the findings from Example 3 using an analytic solution. The results demonstrate excellent accuracy even with a small number of collocation points. For instance, the $L^2$-error of the derivative decreases rapidly from $3.81 \times 10^{-2}$ at $N=4$ to $5.47 \times 10^{-15}$ at $N=24$. 
Similar to the previous smooth case, the convergence rates increase monotonically with $N$, reaching approximately $26.3$ at $N=24$. This behavior is fully consistent with the expected spectral convergence for analytic functions. Meanwhile, the condition number of stiffness matrices grows algebraically but remains within a manageable range, allowing the method to achieve machine-precision accuracy without stability issues.

\begin{figure}[h!]
	\centering
	\includegraphics[width=0.49\textwidth]{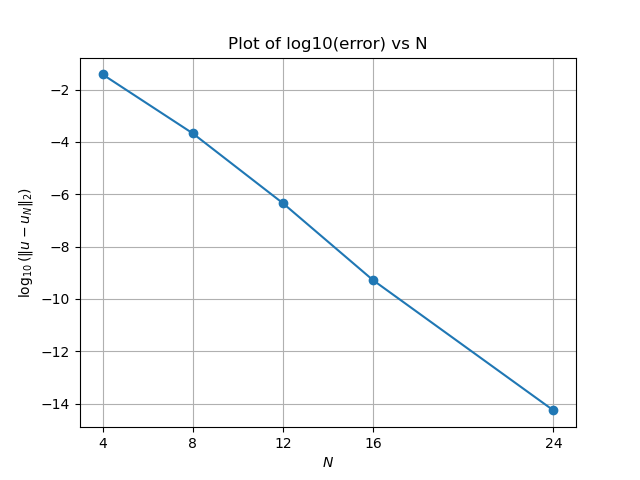}
	\caption{Convergence results depicted with \(\log_{10}(\| u - u_{N}\|_{L^2})\) versus \(N\). }
	\label{fig:example4_convergence}
\end{figure}

Then the approximately linear decay of the curves over a range of \(N\) demonstrates that our proposed collocation method achieves exponential convergence for analytic solutions.

\begin{figure}[H]
	\centering
	
	\begin{subfigure}[t]{0.5\textwidth}
		\centering
		\includegraphics[width=\linewidth]{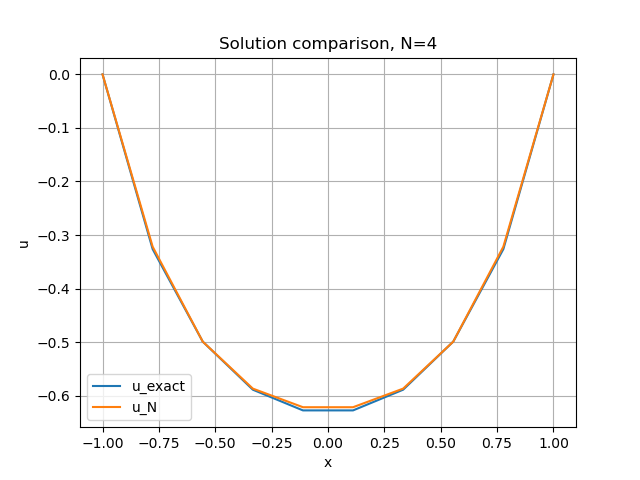}
		\caption{Sketches of $u$ and $u_N$}
	\end{subfigure}\hfill
	\begin{subfigure}[t]{0.5\textwidth}
		\centering
		\includegraphics[width=\linewidth]{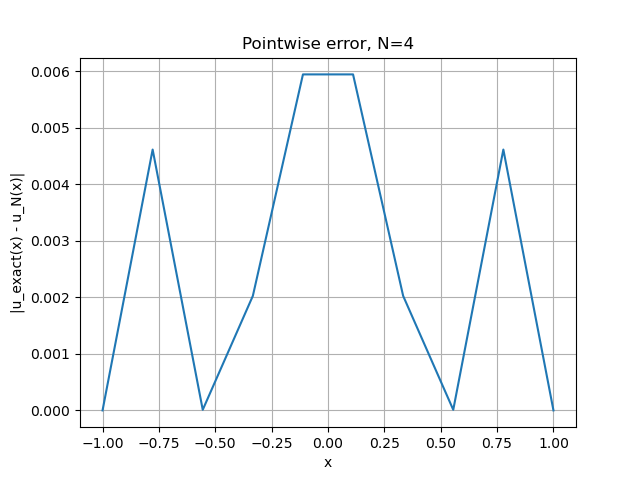}
		\caption{Pointwise error of $|u-u_N|$}
	\end{subfigure}
	
	\caption{Solution profile and pointwise errors for $N=4$. (left: true solution $u$ versus numerical solution $u_N$. right: $|u-u_N|$)}
	\label{fig:example4_sol_err1}
\end{figure}

\begin{figure}[H]
	\centering
	\begin{subfigure}[t]{0.5\textwidth}
		\centering
		\includegraphics[width=\linewidth]{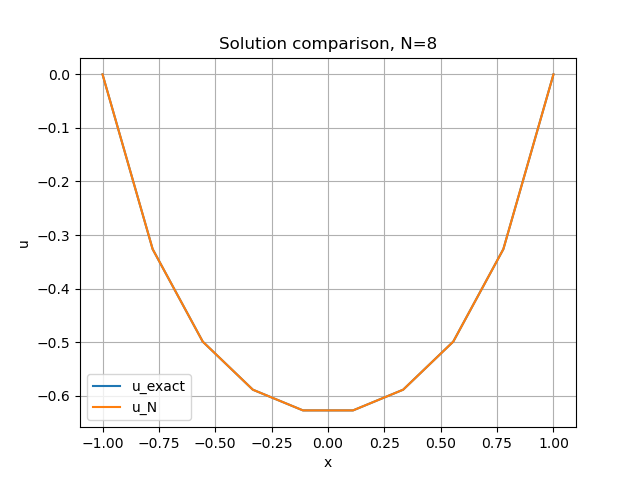}
		\caption{Sketches of $u$ and $u_N$}
	\end{subfigure}\hfill
	\begin{subfigure}[t]{0.5\textwidth}
		\centering
		\includegraphics[width=\linewidth]{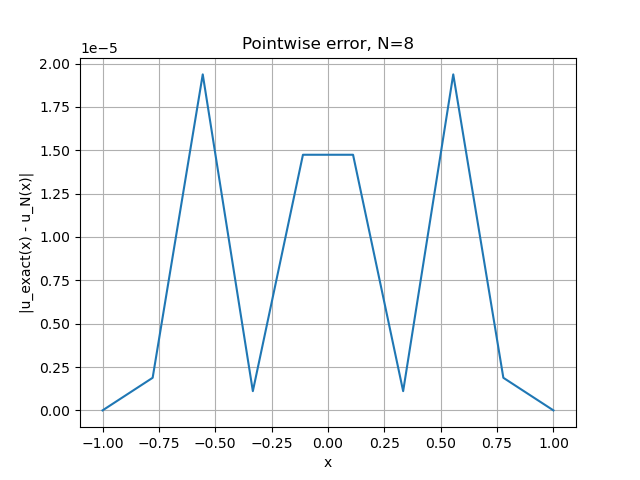}
		\caption{Pointwise error of $|u-u_N|$}
	\end{subfigure}
	
	\caption{Solution profile and pointwise error for $N=8$. (left: true solution $u$ versus numerical solution $u_N$. right: $|u-u_N|$)}
	\label{fig:example4_sol_err2}
\end{figure}

\begin{figure}[H]
	\centering
	\begin{subfigure}[t]{0.5\textwidth}
		\centering
		\includegraphics[width=\linewidth]{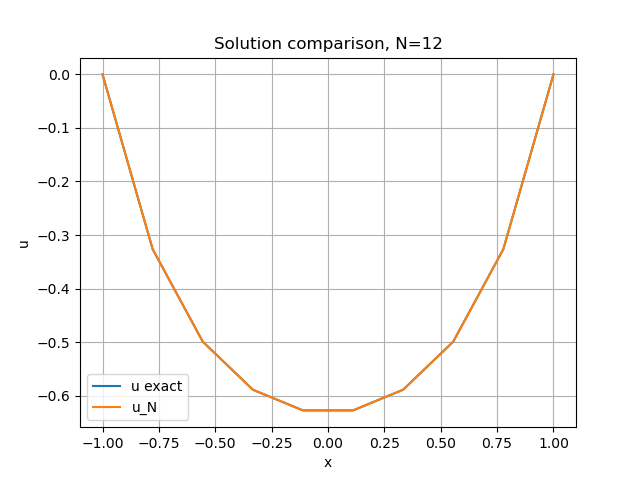}
		\caption{Sketches of $u$ and $u_N$}
	\end{subfigure}\hfill
	\begin{subfigure}[t]{0.5\textwidth}
		\centering
		\includegraphics[width=\linewidth]{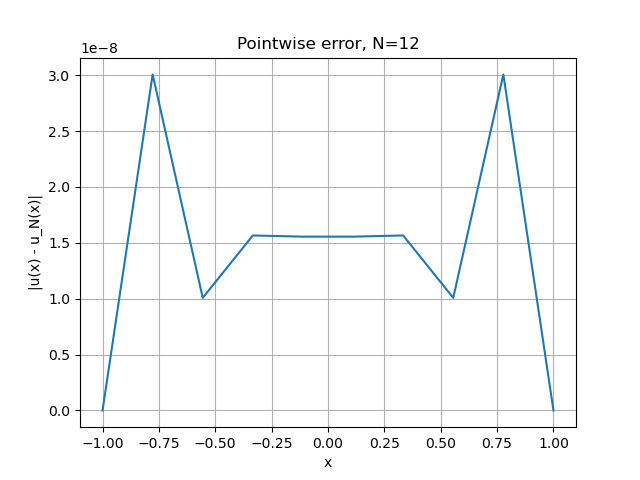}
		\caption{Pointwise error of $|u-u_N|$}
	\end{subfigure}
	
	\caption{Solution profile and pointwise error for $N=12$. (left: true solution $u$ versus numerical solution $u_N$. right: $|u-u_N|$)}
	\label{fig:example4_sol_err3}
\end{figure}

\begin{figure}[H]
	\centering
	\begin{subfigure}[t]{0.5\textwidth}
		\centering
		\includegraphics[width=\linewidth]{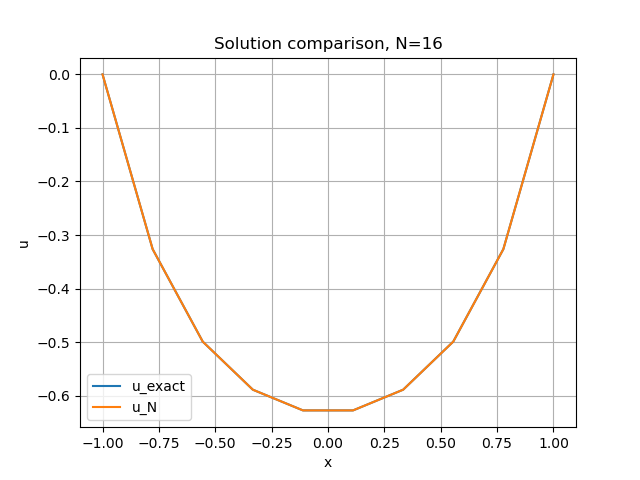}
		\caption{Sketches of $u$ and $u_N$}
	\end{subfigure}\hfill
	\begin{subfigure}[t]{0.5\textwidth}
		\centering
		\includegraphics[width=\linewidth]{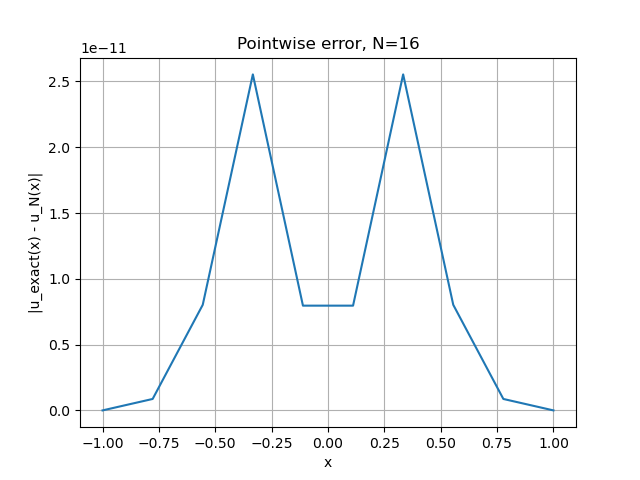}
		\caption{Pointwise error of $|u-u_N|$}
	\end{subfigure}
	
	\caption{Solution profile and pointwise error for $N=16$. (left: true solution $u$ versus numerical solution $u_N$. right: $|u-u_N|$)}
	\label{fig:example4_sol_err4}
\end{figure}

Figure~\ref{fig:example4_convergence} illustrates the convergence results on a semi-logarithmic scale. The strictly linear decay of $\log_{10}(\|u - u_N\|_{L^2})$ versus $N$ provides compelling visual evidence of exponential convergence.
The solution profiles and pointwise errors in Figures~\ref{fig:example4_sol_err1}--\ref{fig:example4_sol_err4} further substantiate the high efficiency of our proposed  method. The numerical approximation $u_N$ captures the analytic solution $u$ with remarkable accuracy. As shown in the right panels, the pointwise error decreases as $N$ increases from $4$ to $16$, with the maximum error uniformly bounded within the interior of the domain.

\section{Conclusion}

We investigate a second-order differential problem using a novel collocation scheme. The basis functions are constructed using Chebyshev polynomials of the first kind, which ensures the sparsity of the stiffness matrices. Error estimates are expressed in two different norms by employing the Riesz theorem. Numerical tests are provided to confirm the theoretical results, demonstrating that the proposed collocation method can deliver highly accurate numerical solutions.

The main contributions of this paper are summarized as follows:
\begin{itemize}
	\item Study an elliptic-type boundary value problem on $I=(-1,1)$ with the endpoint-degenerate coefficient
	$\omega(x)=\sqrt{1-x^2}$ and homogeneous Dirichlet boundary conditions.
	The model is formulated in a natural weighted variational setting, and well-posedness of the weak solution is established by proving continuity and coercivity of the associated bilinear form, based on a weighted Poincar\'e-type inequality and the Lax--Milgram theorem.
	\item Propose a Chebyshev-based discretization with boundary-satisfying basis functions,and derive an explicit and sparse algebraic system by combining the operator identity with Chebyshev orthogonality. 
	\item Introduce an appropriate orthogonal (Ritz-type) projection in the weighted space and derive a priori error estimates in two norms,namely the weighted energy norm and an $L^2$-type norm,thereby quantifying the convergence behavior of the proposed collocation approximation.
	\item Perform comprehensive numerical experiments for both low-regularity and smooth solutions.
	The results include errors, convergence rates, condition numbers, and $N$-$\log$ curves,which corroborate the theoretical error analysis and demonstrate the performance of the proposed method.
\end{itemize}

In the future, we plan to extend similar techniques to other typical differential problems, including degenerate coefficients and high dimensional domains.
 
\section*{Conflict of Interest} The authors declare that they have no conflict of interest.



\end{document}